\documentclass[12pt]{article}
\usepackage{preprint-layout} 
\usepackage{preprint-notation} 
\usepackage{empheq}
\newcommand*\widefbox[1]{\fbox{\hspace{1em}#1\hspace{1em}}}

\title{A Cut Finite Element Method for Elliptic Bulk Problems with Embedded Surfaces}
\date{\today}

\author{
Erik Burman,
Peter Hansbo,
Mats G. Larson and
David Samvin
}
\date{}

\begin{document}

\maketitle

\begin{abstract}
We propose an unfitted finite element method for flow in fractured
porous media. The coupling across the fracture uses a Nitsche type
mortaring, allowing for an accurate representation of the  jump in the
normal component of the gradient of the discrete solution across the 
fracture.
The flow field in the fracture is modelled simultaneously, using the
average of traces of the bulk variables on the fractured. In
particular the Laplace-Beltrami operator for the transport in the
fracture is included using the average of the projection on the
tangential plane of the fracture of the trace of the bulk
gradient. Optimal order error estimates are proven under suitable 
regularity assumptions on the domain geometry.
The extension to the case of bifurcating fractures is discussed. Finally
the theory is illustrated by a series of numerical examples.
\end{abstract}

\section{Introduction}

We consider a model Darcy creeping flow problem with low permeability in the bulk and with embedded 
interfaces with high permeability. Our approach 
is based on the Nitsche extended finite element of Hansbo and Hansbo \cite{HaHa02}, which 
however did not include transport on the interface. Here, we follow  Capatina et al. \cite{CaLuElBa16} and let a suitable mean of the solution on the interface be affected by a transport equation see also \cite{BurClaHan15}. We present a complete a priori analysis and consider the important extension to bifurcating fractures.
   
The flow model we use is essentially the one proposed 
in \cite{CaLuElBa16}. More sophisticated models have been proposed, e.g., 
in \cite{AnBoHu09,FoFuScRu14,FrRoJeSa08,MaJaRo05}, in particular allowing for jumps in the solution
across the interfaces. To allow for such jumps, one can either align the mesh with the interfaces, as in, e.g., \cite{HaAsDaEiHe09}, or use extended finite element techniques, cf. 
\cite{BurClaHan15,CaLuElBa16,DASc12,DeFuSc17}. 

In previous work \cite{BuHaLa17b} we used a continuous approximation with the interface equations simply added to the bulk equation, which
does not allow for jumps in the solution. This paper presents a more
complex but more accurate discrete solution to the problem. To reduce
the technical detail of the arguments we
consider a semi-discretization of the problem where we assume that the
integrals on the interface and the subdomains separated by the
interface can be evaluated exactly. The results herein can be extended
to the fully discrete setting, with a piecewise affine approximation
of the fracture using the analysis detailed in \cite{BurHanLarZah16}.

An outline of the paper is as follows: In Section \ref{modelp} we formulate the model problem, 
its weak form, 
and investigate the regularity properties of the solution, in Section \ref{femp} we formulate 
the finite element method, in Section \ref{errp} we derive error estimates, in Section \ref{bif} 
we extend the approach to the case of bifurcating fractures, and in Section \ref{numex} we present numerical examples including a study of the convergence and a more 
applied example with a network of fractures.

\section{The Model Problem\label{modelp}}
In this section we introduce our modelproblem. First we present the
strong form of the equations
and then we derive the weak form that is used for the finite element
modelling. We discuss the regularity properties of the solution
and show that if the fracture is sufficiently smooth the problem
solution, restricted to the subdomains partitioning the global domain, has a regularity that allows for optimal approximation
estimates for piecewise affine finite element methods.
\subsection{Strong and Weak Formulations}
Let $\Omega$ be a convex polygonal domain in $\IR^d$, 
with $d=2$ or $3$. Let $\Gamma$ be a smooth embedded 
interface in $\Omega$, which partitions $\Omega$ into 
two subdomains $\Omega_1$ and $\Omega_2$. We 
consider the problem: find $u:\Omega \rightarrow \IR$ 
such that 
\begin{alignat}{3}\label{eq:strong-a}
-\nabla \cdot a \nabla u &= f &\qquad &\text{in $\Omega_i$, $i=1,2$}
\\  \label{eq:strong-b}
-\nabla_\Gamma \cdot a_\Gamma  \nabla_\Gamma u_\Gamma &= f_\Gamma 
- \llbracket n \cdot a \nabla u \rrbracket & \qquad &\text{on $\Gamma$}
\\ \label{eq:strong-c}
[u] &= 0 & \qquad &\text{on $\Gamma$}
\\ \label{eq:strong-d}
u &=0 & \qquad  &\text{on $\partial \Omega$}
\end{alignat}
Here 
\begin{equation}
[v] = v_1 - v_2, 
\qquad 
\llbracket n \cdot a \nabla v \rrbracket = n_1 \cdot a_1 \nabla v_1 +  n_2 \cdot a_2 \nabla v_2
\end{equation}
where $v_i = v |_{H^1(\Omega_i)}$, $n_i$ is the exterior unit normal to $\Omega_i$, $a_i$ are positive bounded 
permeability coefficients, for simplicity taken as constant, and $0\leq a_\Gamma <\infty$ is a constant permeability coefficient on the interface. Note that it follows from (\ref{eq:strong-c}) that $v$ is continuous across $\Gamma$ 
while from (\ref{eq:strong-c})  we conclude that the  normal flux is in general not 
continuous across $\Gamma$. Note also that taking $a_\Gamma = 0$ and $f_\Gamma=0$ corresponds to a standard Poisson problem with possible jump in permeability coefficient across $\Gamma$.

\begin{figure}
\begin{center}
\includegraphics[scale=0.4]{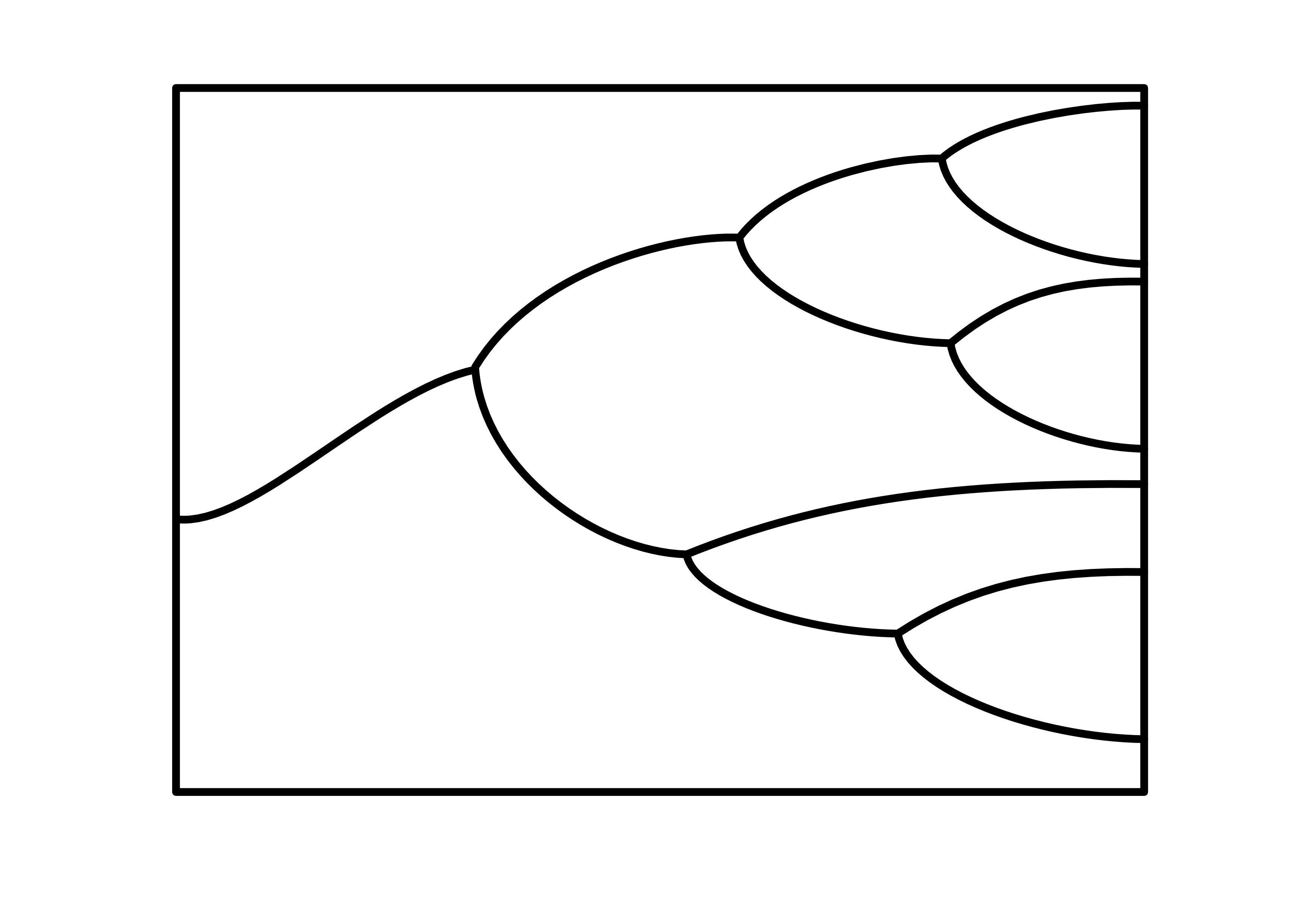}
\end{center}
\caption{Schematic figure of bifurcating fractures.}
\end{figure}
To derive the weak formulation of the system we introduce the
$L^2$-scalar product over a domain $X \subset \mathbb{R}^d$, or $X
\subset \mathbb{R}^{d-1}$. For $u,v \in L^2(X)$ let
\begin{equation}
(u,v)_X = \int_X u\,v~\mbox{d}X
\end{equation}
with the associated norm $\|u\|_X = (u,u)_X^{1/2}$. Multiplying 
(\ref{eq:strong-a}) by $v \in V = H^1(\Omega)\cap H^1(\Gamma)$, 
integrating by parts over $\Omega_i$, and using (\ref{eq:strong-b}) we obtain
\begin{align}\label{eq:weak-der-a}
(f,v)_\Omega 
&= -(\nabla \cdot a \nabla  u, v)_{\Omega_1} 
-(\nabla \cdot a \nabla  u, v)_{\Omega_2}
\\ \label{eq:weak-der-b}
&=(a \nabla u,\nabla v)_{\Omega_1}  + (a \nabla u,\nabla v)_{\Omega_2}
      - (\llbracket n \cdot a \nabla u \rrbracket, v )_\Gamma
\\ \label{eq:weak-der-c}
&= (a \nabla u,\nabla v)_\Omega 
- (f_\Gamma + \nabla_\Gamma \cdot a_\Gamma \nabla_\Gamma u, v )_\Gamma
\\ \label{eq:weak-der-d}
&= (a \nabla u,\nabla v)_\Omega 
+ ( a_\Gamma \nabla_\Gamma u, \nabla_\Gamma v )_\Gamma 
- (f_\Gamma,v)_\Gamma
\end{align}
We thus arrive at the weak formulation: find $u \in V$ such that 
\begin{equation}\label{eq:weak-problem}
 (a \nabla u,\nabla v)_\Omega 
+ ( a_\Gamma \nabla_\Gamma u, \nabla_\Gamma v )_\Gamma 
=
(f,v)_\Omega +  (f_\Gamma,v)_\Gamma\qquad \forall v \in V
\end{equation}
Observing that $V$ is a Hilbert space with scalar product 
\begin{equation}
a(v,w) = (a \nabla v,\nabla w)_\Omega
+ ( a_\Gamma \nabla_\Gamma v, \nabla_\Gamma w)_\Gamma
\end{equation}
and associated norm $\| v \|^2_a = a(v,v)$ it follows from the 
Lax-Milgram Lemma that there is a unique 
solution to (\ref{eq:weak-problem}) in $V$ for 
$f\in H^{-1}(\Omega)$ and $f_\Gamma \in H^{-1}(\Gamma)$.

\subsection{Regularity Properties}
To prove optimality of our finite element method we need that the exact 
solution is sufficient is sufficiently regular. However since the normal fluxes 
jumps over the interface the solution can not have square integrable weak 
second derivatives. If the interface is smooth however we will prove that the
solution restricted to the different subdomains $\Omega_1$, $\Omega_2$
and $\Gamma$ is regular. The upshot of the unfitted finite element is
that this local regularity is sufficient for optimal order approximation.
More precisely  we have the elliptic regularity estimate 
\begin{empheq}[box=\widefbox]{equation}\label{eq:ell-reg}
\| u \|_{H^2(\Omega_1)} + \| u \|_{H^2(\Omega_2)}  + \| u \|_{H^2(\Gamma)} 
\lesssim    
\| f \|_\Omega + \| f_\Gamma \|_\Gamma
\end{empheq}
\begin{proof}
Let $u_i \in H^1_0(\Omega_i)$ solve 
\begin{equation}
(a_i \nabla u_i, \nabla v)_{\Omega_i} = (f,v)_{\Omega_i}\qquad \forall v \in H^1_0(\Omega_i)
\end{equation}
Then we have 
\begin{equation}\label{eq:ell-reg-sub}
\| u_i \|_{H^2(\Omega_i)} \lesssim \| f \|_{\Omega_i} \qquad i = 1,2
\end{equation}
Writing $u = u_\Gamma + u_1 + u_2$ where $u_\Gamma$ 
satisfies
\begin{align}\label{eq:u-Gamma-a}
-\nabla_\Gamma \cdot a_\Gamma \nabla_\Gamma u _\Gamma 
&= f_\Gamma + \llbracket n \cdot a \nabla (u_\Gamma + u_1 + u_2)\rrbracket
\\ \label{eq:u-Gamma-b}
&= f_\Gamma + \llbracket n \cdot a \nabla u_\Gamma \rrbracket 
+ n_1 \cdot a \nabla  u_1 + n_2 \cdot a \nabla u_2
\qquad \text{on $\Gamma$}
\end{align}
and 
\begin{equation}
-\nabla \cdot a \nabla u_\Gamma = 0 \qquad \text{on $\Omega_i$, $i=1,2$}
\end{equation}
Using (\ref{eq:ell-reg-sub}) we conclude that 
\begin{equation}
n_i \cdot a \nabla u_i  |_\Gamma \in H^{1/2}(\Gamma)\qquad i=1,2
\end{equation}
Furthermore, using that $u_\Gamma \in H^1(\Gamma)$, which follows from the fact that 
$u_\Gamma \in V$  it follows that  $u_\Gamma |_{\Omega_i}  \in H^{3/2}(\Omega_i)$, $i=1,2,$ 
and thus 
\begin{equation}
\llbracket n \cdot a \nabla u_\Gamma \rrbracket  \in H^{1/2}(\Gamma)
\end{equation}
Since the right hand side of (\ref{eq:u-Gamma-b}) is in $L^2(\Gamma)$
we may use elliptic regularity for the Laplace Beltrami operator to confirm that 
\begin{equation}
u_\Gamma|_\Gamma \in H^2(\Gamma)
\end{equation}
Collecting 
the bounds we obtain the refined regularity estimate 
\begin{equation}\label{eq:regularity}
\| u_\Gamma \|_{H^2(\Gamma)} 
+ 
\sum_{i=1}^2 \left( \| u_\Gamma \|_{H^{5/2}(\Omega_i)} + \| u_i \|_{H^2(\Omega_i)} \right)
\lesssim 
\| f \|_\Omega + \| f_\Gamma \|_\Gamma
\end{equation} 
where we note that we have stronger control of $u_\Gamma$ on the subdomains.
\end{proof}
\section{The Finite Element Method\label{femp}} 

\subsection{The Mesh and Finite Element Space}
Let $\mcT_h$ be a quasi uniform conforming mesh, consisting of 
shape regular elements with mesh parameter $h\in(0,h_0]$, on $\Omega$ 
and let 
\begin{equation}
\mcT_{h,i} = \{ T \in \mcT_h : T \cap \Omega_i \neq \emptyset\}
\qquad
i=1,2
\end{equation}
be the active meshes associated with $\Omega_i$, $i=1,2.$
Let $V_h$ be a finite element space consisting of continuous 
piecewise polynomials on $\mcT_h$ and define
\begin{equation}
V_{h,i} = V_h |_{\mcT_{h_i}} \qquad i =1,2
\end{equation} 
and 
\begin{equation}
W_h = V_{h,1} \oplus V_{h,2}
\end{equation}
To $v =v_1 \oplus v_2 \in W_h$ we associate the function $\widetilde{v} \in L^2(\Omega)$ 
such that $\widetilde{v}|_{\Omega_i} = v_i|_{\Omega_i}$, $i=1,2$. In general, we simplify the 
notation and write $\widetilde{v} = v$.

\subsection{Derivation of the Method}
To derive the finite element method we follow the same approach as
when introducing the weak formulation, but taking care to handle the
boundary integrals that appear due to the discontinuities in the
approximation space.

Testing the exact problem with $v \in W_h$ and integrating by parts 
over $\Omega_1$ and $\Omega_2$ we find that 
\begin{align}
&(f,v)_{\Omega_1} + (f,v)_{\Omega_2}
\\
&\qquad=
(-\nabla \cdot a \nabla u,v)_{\Omega_1} 
+ 
(-\nabla \cdot a \nabla u,v)_{\Omega_2} 
\\
&\qquad=
(a\nabla u,\nabla v)_\Omega 
- (\langle n\cdot a \nabla u \rangle, [v] )_\Gamma 
- (\llbracket n\cdot a \nabla u\rrbracket,\langle v \rangle_* )_\Gamma
\\
&\qquad=  (a\nabla u,\nabla v)_\Omega 
- (\langle n\cdot a \nabla u \rangle, [v] )_\Gamma
- (\nabla_\Gamma \cdot a_\Gamma \nabla_\Gamma u,\langle v \rangle_*)_\Gamma 
- (f_\Gamma,\langle v \rangle_* )_\Gamma
\\
&\qquad= (a \nabla u,\nabla v)_\Omega 
- (\langle n\cdot a \nabla u \rangle, [v] )_\Gamma
+ (a_\Gamma \nabla_\Gamma u,\nabla_\Gamma \langle v \rangle_* )_\Gamma 
- (f_\Gamma,\langle v \rangle_* )_\Gamma
\\
&\qquad= (a \nabla u,\nabla v)_\Omega 
- (\langle n\cdot a \nabla u \rangle, [v] )_\Gamma
- ([u],\langle n\cdot a \nabla v \rangle )_\Gamma
\\ 
&\qquad \qquad
+ (a_\Gamma \nabla_\Gamma u,\nabla_\Gamma \langle v \rangle_* )_\Gamma
- (f_\Gamma,\langle v \rangle_* )_\Gamma
\end{align}
where in the last identity we symmetrized using the fact that $[u]=0$. We also 
used the identity 
\begin{equation}
[v w ] = [v]\langle w \rangle + \langle v \rangle_* [w]
\end{equation}
where the averages are defined by
\begin{equation}
\langle w \rangle = \kappa_1 w_1 + \kappa_2 w_2,
\qquad 
\langle w \rangle_* = \kappa_2 w_1 + \kappa_1 w_2
\end{equation}
with $\kappa_1+\kappa_2=1$ and $0\leq \kappa_i \leq 1$.

Introducing the bilinear forms
\begin{gather}
a_{\Omega}(v,w)= (a \nabla v,\nabla w)_{\Omega_1}+(a \nabla v,\nabla w)_{\Omega_2}
- (\langle n\cdot a \nabla v \rangle, [w] )_\Gamma
- ([v],\langle n\cdot a \nabla w \rangle )_\Gamma
\\
a_{h,\Gamma}(v,w) 
= (a_\Gamma \nabla_\Gamma \langle v \rangle_*,\nabla_\Gamma \langle w \rangle_* )_\Gamma,
\\
 l_h(v) = (f,v)_\Omega + (f_\Gamma,\langle v \rangle_*)_\Gamma
\end{gather}
the above formal derivation leads to the following consistent formulation for
discontinuous test functions $w$. For $u\in W = H^1(\Omega)\cap H^{3/2}(\Omega_1) 
\cap H^{3/2}(\Omega_2) \cap H^1(\Gamma)$ the solution to (\ref{eq:weak-problem}) there holds 
\begin{empheq}[box=\widefbox]{equation}\label{eq:formulation}
a_{\Omega}(u,w)+ a_{h,\Gamma}(u,w) = l_h(w) \qquad \forall w \in W_h
\end{empheq}
Observe that we have modified $a_{h,\Gamma}$ by introducing the
average $\langle v \rangle_*$ also in the left factor. This changes
nothing when applied to a smooth solution, but will allow also to
apply the form to the discontinuous discrete approximation space. The
subscript $h$ in the form indicates that it is modified to be well
defined for the discontinuous approximation space. The definition of $W$ is 
motivated by the fact that the trace terms should be well defined, for instance, 
\begin{align}
(\langle n \cdot a \nabla v \rangle, [w] )_\Gamma 
&\lesssim 
\left( \sum_{i=1}^2 \| v_i \|^2_{H^1(\partial \Omega_i)} \right)^{1/2} 
\left( \sum_{i=1}^2 \| w_i \|^2_{\partial \Omega_i} \right)^{1/2}
 \\
& \lesssim 
\left( \sum_{i=1}^2 \| v_i \|^2_{H^{3/2}( \Omega_i)} \right)^{1/2} 
\left( \sum_{i=1}^2 \| w_i \|^2_{H^1(\Omega_i)} \right)^{1/2}
\end{align}
where we used the trace inequalities $\| v \|_{H^s(\partial \Omega_i)} \lesssim \| v \|_{H^{s+1/2}(\Omega_i)}$ for $s>0$ and $\|w \|_{\partial \Omega_i} 
\lesssim \| w \|_{H^{1/2 + \epsilon}(\Omega_i)} \lesssim \| w \|_{H^1(\Omega_i)}$ for $\epsilon>0$.
\subsection{The Finite Element Method} 
The finite element method that we propose is based on the formulation
\eqref{eq:formulation}. However, using this formulation as it stands
does not lead to a robust approximation method. Indeed we need to ensure stability of the
formulation through the addition of consistent penalty terms. First we
need to enforce continuity of the discrete solution across
$\Gamma$. To this end we introduce an augmented version of $a_\Omega$,
\[
a_h(v,w)=
a_\Omega(v,w)
+ \beta h^{-1} ([v],[w])_\Gamma
\]
with $\beta$ a positive parameter. Since the exact solution $u \in
H^1(\Omega)$, there holds $a_\Omega(u,w) = a_h(u,w)$.
Secondly, to obtain stability independently of how the interface cuts the
computational mesh and for strongly varying permeabilities $a_1$,
$a_2$ and $a_\Gamma$ we also need some penalty terms in a
neighbourhood of the interface. We define
\[
s_h(v,w) = s_{h,1}(v_1,w_1)+s_{h,2}(v_2,w_2)
\]
where
\begin{equation}\label{eq:stab}
s_{h,i}(v_i,w_i)=\gamma h([n \cdot a \nabla v_i],[n \cdot a \nabla w_i])_{\mathcal{F}_{h,i}}\qquad i=1,2  
\end{equation}
where $\gamma$ is a positive parameters
and $\mathcal{F}_{h,i}$ is the set of interior faces in $\mcT_{h,i}$  
that belongs to an element $T\in \mcT_{h,i}$ which intersects 
$\Gamma$, see Fig. \ref{fig:edges}. Observe that for $u \in H^2(\Omega_1\cup\Omega_2)$,
$s_h(u,v)=0$ for all $v \in W_h$.

Collecting the above bilinear forms in 
\begin{equation}\label{eq:Ah}
A_h(v,w)
= a_h(v,w) + s_h(v,w) + a_{h,\Gamma}(v,w)
\end{equation}
the finite element method reads: 
\begin{empheq}[box=\widefbox]{equation}\label{eq:FEM}
\mbox{Find $u_h \in W_h$ such that:  } A_h(u_h,v) = l_h(v)\qquad \forall v \in W_h
\end{empheq}

\section{Analysis of the Method\label{errp}}
In this section we derive the basic error estimates that the solution
of the formulation \eqref{eq:FEM} satisfies. The technical detail is
kept to a minimum to improve readability. In particular, we assume that 
the bilinear forms can be computed exactly and that $\Gamma$ fulfils the conditions of \cite{HaHa02}. For a more complete
exposition in a similar context we refer to \cite{BurHanLarZah16}.

\subsection{Properties of the Bilinear Form}
For the analysis it is convenient to define the following energy norm 
\begin{equation}
\tn v \tn_h^2 
= \sum_{i=1}^2(\| a_i^{1/2} \nabla v \|_{\Omega_i}^2
+ | v|^2_{s_i})+
c_a h \| \langle n \cdot a \nabla v \rangle \|^2_{\Gamma} 
+ 
\beta h^{-1} \| [v ] \|^2_{\Gamma}
+
\| a_\Gamma\nabla_\Gamma \langle v \rangle_* \|^2_{\Gamma}
\end{equation}
where $| v|_{s_i}=s_i(v,v)^{1/2}$.
\begin{lem}
The form $A_h$, defined in (\ref{eq:Ah}), satifies the following bounds:
\begin{itemize}
\item $A_h$ is continuous 
\begin{empheq}[box=\widefbox]{equation}\label{eq:continuity-Ah}
A_h(v,w) \lesssim \tn v \tn_h \tn w \tn_h \qquad 
v,w \in  W + W_h 
\end{empheq}
where $W$ was introduced in (\ref{eq:formulation}). 
\item $A_h$ is coercive on $W_h$,
\begin{empheq}[box=\widefbox]{equation}\label{eq:coercivity-Ah}
\tn v \tn_h^2 \lesssim A_h(v,v)\qquad v \in W_h
\end{empheq} 
provided $\beta$ is large enough. 
\end{itemize}
\end{lem}
\begin{proof}  The first estimate (\ref{eq:continuity-Ah}) follows directly 
from the Cauchy-Schwarz inequality. To show (\ref{eq:coercivity-Ah}) 
we recall the following inequalities:
\begin{equation}\label{eq:ghost_stab}
\|a_i^{1/2} \nabla v \|^2_{\mcT_{h,i}} \lesssim
\|a_i^{1/2}\nabla v \|^2_{\Omega_i} + | v |^2_{s_{h,i}}
\quad \mbox{ (see \cite{BH12})}
\end{equation}
\begin{equation}\label{eq:trace_ineq}
h \|\langle n \cdot a \nabla v \rangle \|^2_\Gamma 
\lesssim 
\sum_{i=1}^2 \|\kappa_i a_i  \nabla v \|^2_{\mcT_{h,i}(\Gamma)} \mbox{
\quad  (see \cite{HaHa02})}
\end{equation}
In \eqref{eq:trace_ineq} we used the notation $\mcT_{h,i}(\Gamma):=
\{T \in \mcT_{h,i}: T \cap \Gamma \ne \emptyset\}$.
To prove the claim observe that for all $v \in W_h$ 
\begin{equation}
A_h(v,v) = \sum_{i=1}^2(\| a_i^{1/2} \nabla v \|_{\Omega_i}^2
+ | v|^2_{s_i}) +  
\beta h^{-1} \| [v ] \|^2_{\Gamma}
+
\| a_\Gamma\nabla_\Gamma \langle v \rangle_* \|^2_{\Gamma}-2(\langle n\cdot a \nabla v \rangle, [v])_\Gamma
\end{equation}
Using \eqref{eq:ghost_stab} and \eqref{eq:trace_ineq} we obtain the
following bound on the fluxes
\begin{equation}\label{eq:flux_bound}
h \| \langle n \cdot a \nabla v \rangle \|^2_{\Gamma} \leq
C \sum_{i=1}^2 \kappa_i a_i (\|a_i^{1/2}\nabla v \|^2_{\Omega_i} + | v |^2_{s_{h,i}})
\end{equation}
Now assume that $\kappa_i a_i \leq a_{min}:=\min_{i\in \{1,2\}} a_i$, for instance one may take
$\kappa_1 = a_2/(a_1+a_2)$ and $\kappa_2=a_1/(a_1+a_2)$ then 
\begin{align}
2(\langle n\cdot a \nabla v \rangle, [v])_\Gamma 
&\leq 2
a_{min}^{-1/2} h^{1/2} \| \langle n \cdot a \nabla v \rangle
\|_{\Gamma} a_{min}^{1/2}  h^{-1/2} \|[v]\|_\Gamma\\[3mm]
 &\leq \varepsilon h
a_{min}^{-1} \| \langle n \cdot a \nabla v \rangle
\|^2_{\Gamma}+ a_{min} h^{-1} \varepsilon^{-1} \|[v]\|^2_\Gamma 
\\
&\leq C \varepsilon \sum_{i=1}^2 (\|a_i^{1/2}\nabla v \|^2_{\Omega_i} + | v |^2_{s_{h,i}}) + a_{min} h^{-1} \varepsilon^{-1} \|[v]\|^2_\Gamma
\end{align}
It follows that
\begin{equation}
A_h(v,v) \ge  (1 - C \varepsilon ) \sum_{i=1}^2(\| a_i^{1/2} \nabla v \|_{\Omega_i}^2
+ | v|^2_{s_i}) +  
(\beta - a_{min}/\varepsilon)  h^{-1} \| [v ] \|^2_{\Gamma}
+
\| a^{1/2}_\Gamma \nabla_\Gamma \langle v \rangle_* \|^2_{\Gamma}
\end{equation}
The bound \eqref{eq:coercivity-Ah} now follows taking $\varepsilon =
1/(2C)$ and $\beta > 2 C a_{min}$ and by applying once again
\eqref{eq:flux_bound}, taking $c_a \sim a_{\min}^{-1}$.
\end{proof}

A consequence of the bound \eqref{eq:coercivity-Ah} is the existence
of a unique solution to \eqref{eq:FEM}.

\begin{lem}
The linear system defined by the formulation \eqref{eq:FEM} is invertible.
\end{lem}
\begin{proof}
Follows from Lax-Milgram's lemma.
\end{proof}

\subsection{Interpolation} 

For $\delta>0$ let 
$E_i:H^s(\Omega_i) \rightarrow H^s(\Omega)$ be 
a continuous extension operator $s>0$. We define 
the interpolation operator 
\begin{equation}
\pi_h:L^2(\Omega) \ni v \mapsto \pi_{h,1} v_1 
\oplus 
\pi_{h,2} v_2 \in V_{h,1} \oplus V_{h,2}  = W_h
\end{equation}
where $\pi_{h,i}: L^2(\Omega_i) v_i \mapsto 
\pi^{SZ}_{h,i} E v_i \in V_{h,i}$, $i=1,2,$ and 
$\pi_h^{SZ}$ is the Scott-Zhang interpolation operator. 
We then have the interpolation error estimate
\begin{empheq}[box=\widefbox]{equation}\label{eq:interpol-energy}
\tn u - \pi_h u \tn_h \lesssim 
h \Big( \| u \|_{H^2(\Omega_1)} 
+ \| u \|_{H^2(\Omega_1)}
+ \| u \|_{L^\infty_{\delta_h}(H^2(\Gamma_t))} \Big)
\end{empheq} 
where 
\begin{equation}
\Gamma_t = \{x \in \Omega : \rho_\Gamma(x) = t \},
\qquad |t|\leq \delta_0
\end{equation}
and 
\begin{equation}
\| v \|_{L^\infty_\delta(H^s(\Gamma_t))} 
= 
\sup_{|t|\leq \delta } \| v \|_{H^s(\Gamma_t))}
\end{equation}
\begin{proof}
To prove the estimate \eqref{eq:interpol-energy} we use a
trace-inequality on functions in $H^1(\mcT_h(\Gamma))$ (i.e., with $\|\cdot\|_{\mcT_h(\Gamma)}$ the
broken $H^1$-norm over the elements intersected by $\Gamma$),
\begin{equation}\label{eq:trace_gamma}
\|v_i\|_\Gamma \lesssim h^{-1/2} \|E_i v_i\|_{\mcT_h(\Gamma)} +
h^{1/2} \|\nabla E_i v_i\|_{\mcT_h(\Gamma)}
\end{equation}
see \cite{HaHa02}, then interpolation on $\mcT_h(\Gamma)$ and finally we use the stability of the
extension operator $E_i$.
First observe that
by using the trace inequality \eqref{eq:trace_gamma} we obtain, with $v=u - \pi_h u$
\begin{multline}
\sum_{i=1}^2\| a_i^{1/2} \nabla v_i \|_{\Omega_i}
+\|(\beta h)^{-1/2} [v]\|_\Gamma+
c_a h \| \langle n \cdot a \nabla v \rangle \|_{\Gamma} \\
\lesssim \sum_{i=1}^2\left(h^{-1} \|  v_i \|_{\mcT_h(\Gamma)}+\|
 \nabla v_i \|_{\mcT_h(\Gamma))}+h \|  \nabla^2 v_i \|_{\mcT_h(\Gamma))}\right)
\end{multline}
Using standard interpolation for the Scott-Zhang interpolation
operator we get the bound
\begin{multline}
\sum_{i=1}^2\| a_i^{1/2} \nabla v_i \|_{\Omega_i}
+\|(\beta h)^{-1/2} [v]\|_\Gamma+
c_a h \| \langle n \cdot a \nabla v \rangle \|_{\Gamma} 
\\
\lesssim 
h \sum_{i=1}^2 | E_i u_i|_{H^2( \mcT_h(\Gamma))}
\lesssim  
h \sum_{i=1}^2 | u_i|_{H^2(\Omega_i)}
\end{multline}
where we used the stability of the extension operator in the last inequality.
The bound $|u - \pi_h u|_{s_i} \lesssim h
\sum_{i=1}^2 |a_i^{1/2} u_i|_{H^2(\Omega_i)}$ follows similarly using
element wise trace inequalities follows by interpolation (c.f. \cite{BH12}). 
The interpolation error estimate for the terms due to the
Laplace-Beltrami operator on $\Gamma$ is a bit more delicate. We use a trace 
inequality to conclude that 
\begin{align}
\|a^{1/2}_\Gamma \nabla_\Gamma \langle u - \pi_h u \rangle_\star\|_\Gamma^2 
&\lesssim 
\sum_{i=1}^2 
\|a^{1/2}_\Gamma \nabla_\Gamma ( u_i - \pi_{h,i} u_i ) \|_\Gamma^2
\\
&\lesssim 
\sum_{i=1}^2 
h^{-1}\| \nabla ( u_i - \pi_{h,i} u_i) \|_{\mcT_h(\Gamma)}^2
+
h \| \nabla^2 ( u_i - \pi_{h,i} u_i) \|_{\mcT_h(\Gamma)}^2
\\ \label{eq:interp_H2}
&\lesssim 
\sum_{i=1}^2 
h \| \nabla^2 u_i \|_{ \mcT_h(\Gamma)}^2
\\
&\lesssim 
\delta h \| u \|^2_{L^\infty_\delta(H^2(\Gamma_t))}
\end{align}
Observing that we may take $\delta \sim h $ the estimate follows.
\end{proof}
Comparing \eqref{eq:interpol-energy} with \eqref{eq:regularity} we see
that we have a small mismatch between the regularity that we can prove
and that required to achieve optimal convergence. In view of this we
need to assume a slightly more regular solution for the $H^1$-error
estimates below. The sub optimal regularity also interferes in the
$L^2$-error estimates. Here we need to use \eqref{eq:regularity} on
the dual solution and in this case the additional regularity of the estimate
\eqref{eq:interpol-energy} is not available. Instead we need to find
the largest $\zeta \in [0,1]$ such that $\tn u - \pi_h u \tn_h \lesssim 
h^\zeta \sum_{i=1}^2 \| u \|_{H^2(\Omega_i)}$, which will result in a
suboptimality by a power of $1-\zeta$ in the convergence order in the $L^2$-norm. 
Revisiting the analysis above up to \eqref{eq:interp_H2} we see that
\begin{align}\label{eq:interp_subopt}
\tn u - \pi_h u \tn_h 
&\lesssim
h
\sum_{i=1}^2 | u_i|_{H^2(\Omega_i)}
\\ 
&\qquad 
+\sum_{i=1}^2 h^{-1/2}\|  \nabla ( u_i - \pi_{h,i} u_i) \|_{\mcT_h(\Gamma)}
+ h^{1/2} \| \nabla^2 ( u_i - \pi_{h,i} u_i) \|_{\mcT_h(\Gamma)}
\\ 
& \lesssim (h + h^{1/2})
\sum_{i=1}^2 | u_i|_{H^2(\Omega_i)}
\end{align}

\subsection{Error Estimates}

\begin{thm} The following error estimates hold
\begin{empheq}[box=\widefbox]{gather}\label{eq:energy}
\tn u - u_h \tn_h \lesssim h 
\Big( \| u \|_{H^2(\Omega_1)} +  \| u \|_{H^2(\Omega_2)} 
+  \| u \|_{L^\infty_{\delta_h}(H^2(\Gamma_t))}
\Big)
\\
\label{eq:L2}
\| u - u_h \|_\Omega + \| u - u_h \|_\Gamma 
\lesssim h^{3/2} 
\Big( \| u \|_{H^2(\Omega_1)} +  \| u \|_{H^2(\Omega_2)} 
+  \| u \|_{L^\infty_{\delta_h}(H^2(\Gamma_t))}
 \Big)
\end{empheq}
\end{thm}
\begin{proof}{\bf (\ref{eq:energy}).}
Splitting the error and using the interpolation error estimate we have 
\begin{align}
\tn u - u_h \tn_h &\leq \tn u - \pi_h \tn_h 
+ \tn \pi_h u - u_h \tn_h
\end{align}
Using coercivity \eqref{eq:coercivity-Ah}, Galerkin orthogonality and
continuity \eqref{eq:continuity-Ah} the second term can be 
estimated as follows  
\begin{align}
\tn \pi_h u - u_h \tn_h^2 
&\lesssim 
A_h(  \pi_h u - u_h,  \pi_h u - u_h )
\\
&= 
A_h(  \pi_h u - u,  \pi_h u - u_h )
\\
&\leq 
\tn  \pi_h u - u \tn_h \tn \pi_h u - u_h \tn_h
\end{align}
and thus applying the approximation result \eqref{eq:interpol-energy} we conclude that
\begin{align}
\tn u - u_h \tn_h 
&\lesssim \tn  u - \pi_h u  \tn_h
\\
&\lesssim h \Big( \| u \|_{H^2(\Omega_1)} +  \| u \|_{H^2(\Omega_2)} 
+  \| u \|_{L^\infty_{\delta_h}(H^2(\Gamma_t))}\Big)
\end{align}
\paragraph{(\ref{eq:L2}).} Consider the dual problem 
\begin{equation}
A(v,\phi) = ( v, \psi)_\Omega + (v,\psi_\Gamma) \qquad \forall v \in V
\end{equation}
and recall that by \eqref{eq:regularity} we have the elliptic regularity
\begin{equation}\label{eq:ell-reg-dual}
\sum_{i=1}^2 \| \phi \|_{H^2(\Omega_i)} + \| \phi_\Gamma \|_{H^2(\Gamma)} 
\lesssim 
\sum_{i=1}^2 \| \psi_i \|_{\Omega_i} + \| \psi_\Gamma \|_{\Gamma} 
\end{equation}
Setting $v= e = u - u_h$ and using Galerkin orthogonality, followed by
the continuity \eqref{eq:continuity-Ah} and the suboptimal approximation estimate
\eqref{eq:interp_subopt} on $\tn \phi - \pi_h \phi \tn_h$ we get 
\begin{align}
(e,\psi)_\Omega + (e,\psi_\Gamma)_\Gamma 
&=A_h(e,\phi)
\\
&=A_h(e,\phi - \pi_h \phi ) 
\\
&\leq \tn e \tn_h \tn \phi - \pi_h \phi \tn_h
\\
&\lesssim  \tn e \tn_h  h^{1/2}\left( \sum_{i=1}^2 \| \phi \|_{H^2(\Omega_i} + \| \phi_\Gamma \|_{H^2(\Gamma)} \right)
\\
&\lesssim  h^{1/2} \tn e \tn_h 
 \left( \sum_{i=1}^2 \| \psi_i \|_{\Omega_i} + \| \psi_\Gamma \|_{\Gamma} \right).
\end{align}
In the last step we used the elliptic regularity estimate (\ref{eq:ell-reg-dual}) for the dual problem. 
Setting $\psi_ i = e_i / \| e_i\|_{\Omega_i}$ and $\psi_\Gamma = e_\Gamma/\| e_\Gamma \|_{\Gamma}$ estimate (\ref{eq:L2}) follows.
\end{proof}
\begin{rem} 
As noted before the error estimate in the $L^2$-norm is suboptimal
with a power $1/2$. To improve on this estimate we would need to
sharpen the regularities required for the approximation estimate
\eqref{eq:interpol-energy}. This appears to be highly non-trivial
since the interpolation of $u$ and $u_\Gamma$ can not be separated
when both are interpolated using the bulk unknowns. Therefore we did
not manage to exploit the stronger control that we have on the
harmonic extension of $u_\Gamma$ in \eqref{eq:regularity}. Note however that 
if separate fields are used on the fracture and in the bulk domains we 
would recover optimal order convergence in $L^2$.
\end{rem}

\begin{rem} Using the stronger control of the regularity of the harmonic extension 
provided by (\ref{eq:regularity})  we may however establish an optimal order 
$L^2$ error estimate for the solution on $\Gamma$, 
\begin{empheq}[box=\widefbox]{equation}
\label{eq:L2-Gamma}
\| u - u_h \|_\Gamma 
\lesssim h^2
\Big( \| u \|_{H^2(\Omega_1)} +  \| u \|_{H^2(\Omega_2)} 
+  \| u \|_{L^\infty_{\delta_h}(H^2(\Gamma_t))}
 \Big)
\end{empheq}
\end{rem}

\section{Extension to Bifurcating Fractures\label{bif}}
In the case most common in applications, fractures
bifurcate, leading to networks of interfaces in the bulk. It is
straightforward to include this case in the method above and we will
discuss the method with bifurcating fractures below. The analysis can
also be extended under suitable regularity assumptions, but becomes
increasingly technical. We leave the analysis of the methods modelling
flow in fractured media with bifurcating interfaces to future work.  
\subsection{The Model Problem}
\paragraph{Description of the Domain.} Let us for simplicity consider a two dimensional problem with a one dimensional interface. We define the following: 
\begin{itemize}
\item Let the interface $\Gamma$ be described as a planar graph with nodes 
$\mathcal{N} = \{ x_i \}_{i\in I_N}$ and edges $\mathcal{E} =\{\Gamma_j\}_{j\in I_E}$, 
where $I_N$, $I_E$ are finite index sets, and each $\Gamma_j$ is a smooth curve 
between two nodes with indexes $I_N(j)$. Note that edges only meet in nodes.
\item For each 
$i \in I_N$ we let $I_E(i)$ be the set of indexes corresponding to edges for which $x_i$ is 
a node. For each $i \in I_N$ we let $I_E((i)$ be the set of indexes $j$ such that $x_i$ is an 
end point of $\Gamma_j$, see Figure \ref{fig:bifurcating-layers-notation}. 
\item The graph $\Gamma$ defines a partition of $\Omega$ into $N$ subdomains 
$\Omega_i$, $i=1,\dots,N$. 
\end{itemize}

\begin{figure}
\begin{center}
\includegraphics[scale=0.4]{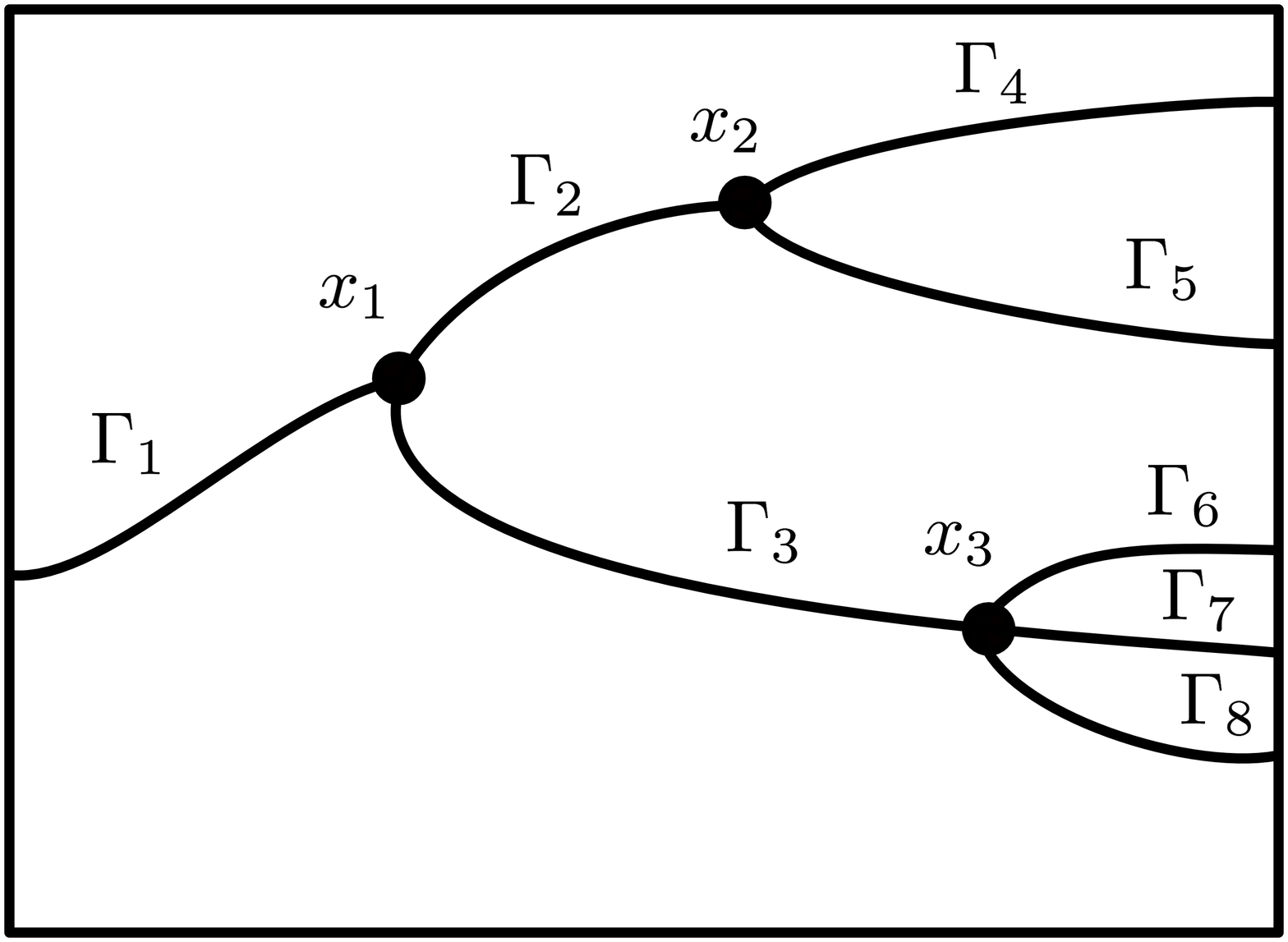}
\end{center}
\caption{Notation for bifurcating fractures.\label{fig:bifurcating-layers-notation}}
\end{figure}

\paragraph{The Kirchhoff Condition.}
The governing equations are given by (\ref{eq:strong-a})--(\ref{eq:strong-d}) 
together with two conditions at each of the nodes $x_i \in \mathcal{N}$, the continuity condition 
\begin{equation}
u_{\Gamma_k} (x_i) = u_{\Gamma_l} (x_i) \qquad \forall k,l \in I_E(i) 
\end{equation}
and the Kirchhoff condition 
\begin{equation}\label{eq:kirchhoff}
\sum_{j \in I_{E}(i)} (t_{\Gamma_j} \cdot a_{\Gamma_j} \nabla_{\Gamma_j} u_{\Gamma_j}) |_{x_j} =0
\end{equation}
where $t_{\Gamma_j} (x_i)$ is the exterior tangent unit vector to $\Gamma_j$ at $x_i$. Note that 
in the special case when a node $x_i$ is an end point of only one curve the Kirchhoff condition 
becomes a homogeneous Neumann condition.

\subsection{The Finite Element Method}

\paragraph{Forms Associated with the Bifurcating Interface.} Let 
$V_\Gamma = \{v \in C(\Gamma) : v \in H^1(\Gamma_j), j \in I_E\}$ 
and  $V = H^1_0(\Omega) \cap V_\Gamma$. We proceed as in the derivation 
(\ref{eq:weak-der-a})--(\ref{eq:weak-der-d}) of the weak problem 
(\ref{eq:weak-problem}) in the standard case. However, when we use 
Green's formula on $\Gamma$ we proceed segment by segment as follows
\begin{align}\nonumber
&\sum_{j \in I_E} -(\nabla_{\Gamma_j} \cdot a_{\Gamma_j} \nabla_{\Gamma_j} u_j,\langle v_j \rangle_*)_{\Gamma_j} 
\\
&\qquad=
\sum_{j \in I_E} ( a_{\Gamma_j} \nabla_{\Gamma_j} u, \nabla_{\Gamma_j} \langle v \rangle_* )_{\Gamma_j} 
- \sum_{j \in I_E} \sum_{i \in I_N (j) } ( t_i \cdot a_{\Gamma_j}\nabla_{\Gamma_j} u,\langle v \rangle_*)_{x_i}
\\
&\qquad=
\sum_{j \in I_E} ( a_{\Gamma_j} \nabla_{\Gamma_j} u,\nabla_{\Gamma_j} \langle v\rangle_*)_{\Gamma_j} 
-\sum_{i \in I_N }   \sum_{j \in I_E(i)} ( t_i \cdot a_{\Gamma_j}\nabla_{\Gamma_j} u,\langle v \rangle_* 
- \langle \langle v\rangle_* \rangle_i)_{x_i}
\end{align} 
where we changed the order of summation and used the Kirchhoff condition (\ref{eq:kirchhoff})  
to subtract the nodal average
\begin{equation}
\langle v \rangle_i = \sum_{j\in I_E(i)} \kappa_j^{\Gamma} v_j(x_i)
\end{equation}
where $0<\kappa_i^{\Gamma}$, and $\sum_{j \in I_E(i)} \kappa_j^{\Gamma} = 1$. Note that when 
a node $x_i$ is an end point of only one curve the contribution from $x_i$ 
is zero, because in that case we have $\langle \langle v\rangle_* \rangle_i|_{x_i} 
- \langle v \rangle_* = 0$ since there is only one element in $I_E(i)$, and thus 
we get the standard weak enforcement of the homogeneous Neumann 
condition.

Symmetrizing and adding a penalty term we obtain the form
\begin{align}
a_{h,\Gamma}(v,w) &= 
\sum_{j \in I_E} ( a_{\Gamma_j} \nabla_{\Gamma_j} \langle v\rangle_*,\nabla_{\Gamma_j} \langle w\rangle_*)_{\Gamma_j} 
\\ \nonumber
&\qquad -\sum_{i \in I_N }   \sum_{j\in I_E(i)} ( t_j \cdot a_{\Gamma_j}\nabla_{\Gamma_j} \langle v\rangle_*
,\langle w \rangle_* - \langle \langle v\rangle_* \rangle_i)_{x_i}
\\ \nonumber
&\qquad -\sum_{i \in I_N }   \sum_{j \in I_E(i)} (\langle v \rangle_* 
- \langle \langle v\rangle_* \rangle_i,  t_j \cdot a_{\Gamma_j}\nabla_{\Gamma_j} \langle w \rangle)_{x_i}
\\ \nonumber
&\qquad + \sum_{i \in I_N }   \sum_{j\in I_E(i)} \beta^{\Gamma} h^{-1} 
( \langle v \rangle_* - \langle \langle v\rangle_* \rangle_i
,
\langle w \rangle_* - \langle \langle w\rangle_* \rangle_i)_{x_i}
\end{align}
where $\beta^\Gamma$ is a stabilisation parameter with the same function as $\beta$.
A similar derivation can be performed for a two dimensional bifurcating 
fracture embedded into $\IR^3$, see \cite{HanJonLarLar17} for further details. 

To ensure coercivity we add a stabilization term of the form 
\begin{equation}
s_{h,\Gamma}(v,w) = \sum_{j \in I_E} s_{h,\Gamma_j}(v,w)
\end{equation}
where 
\begin{equation}
s_{h,\Gamma_j}(v,w) 
= ([\nabla_{\Gamma_j} \langle v \rangle_*] 
, [ \nabla_{\Gamma_j} \langle w\rangle_*] )_{\mcX_h(\Gamma_j)}
\end{equation}
and $\mcX_h(\Gamma_j)$ is the set of points
\begin{equation}
\Gamma_j  \cap \mcF_h(x_i)
\end{equation}
where $\mcF_h(x_i)$ is the set of interior faces in the patch of elements 
$\mcN_h(T(x_i))$ and $T(x_i)$ is an element such that $x_i \in T$. 

We finally define the form $A_{h,\Gamma}$ associated with the bifurcating crack 
by
\begin{equation}
A_{h,\Gamma}(v,w) = a_{h,\Gamma}(v,w) + s_{h,\Gamma} (v,w ) \qquad \forall v \in W_h
\end{equation}

\paragraph{The Method.} Define
\begin{equation}
W_h = \bigoplus_{i=1}^N V_{h,i}
\end{equation}
where $V_{h,i} = V_h|_{\mcT_{h,i}}$. The method takes the form: find $u_h \in W_h$ such that 
\begin{equation}
A_h(u_h,v) = l_h(v)\quad\forall v \in W_h
\end{equation}
where 
\begin{equation}
A_h(v,w) = \sum_{i=1}^N A_{h,i}(v,w) + A_{h,\Gamma}(v,w)
\end{equation}
and 
\begin{equation}
A_{h,i} (v,w) = a_{h,i}(v,w) + s_{h,i}(v,w)
\end{equation}

\section{Numerical Examples\label{numex}}

\subsection{Implementation Details}

We will employ piecewise linear triangles and extend the implementation approach proposed in 
\cite{HaHa02} to include also bifurcating fractures. Recall that  \(\mcT_h(\Gamma) \) denotes the
set of elements intersected by \(\Gamma \), where each side of the
intersection belongs to \(\Omega_1\) and \(\Omega_2\), respectively. For each element in \( T_i \in \mcT_h(\Gamma) \), we assign elements \(T_{i,1} \in \mcT_{h,1}\) and \(T_{i,2} \in \mcT_{h,2}\) by overlapping the existing element \(T_i \in \mcT_h(\Gamma)\) using the \textit{same} nodes from the original triangulation. Elements \(T_{i,1} \) and \(T_{i,2} \) coincide geometrically, see Figure \ref{TriSplit1}. To ensure continuity, we used the same process on the neighboring elements and checked if new nodes had already been assigned. 
For each bifurcation point, two approaches can be adapted. Either by letting the bifurcation point coincide with a node or by the less straight-forward approach to overlap the existing element \(T_i \in \mcT_h(\Gamma) \) into \(T_{i,1} \), \(T_{i,2} \) and \(T_{i,3} \), see Figure \ref{TriSplit2}. For simplicity of implementation, we have here chosen to let the bifurcating point coincide with a node. The triangles \(T_i \notin \mcT_h(\Gamma)\) were handled in the usual way. The stabilization (\ref{eq:stab}) was only applied to the {\em cut sides}\/ of the elements which in all examples was sufficient for stability.

\begin{figure}
	\begin{center}
		\includegraphics[scale=0.4]{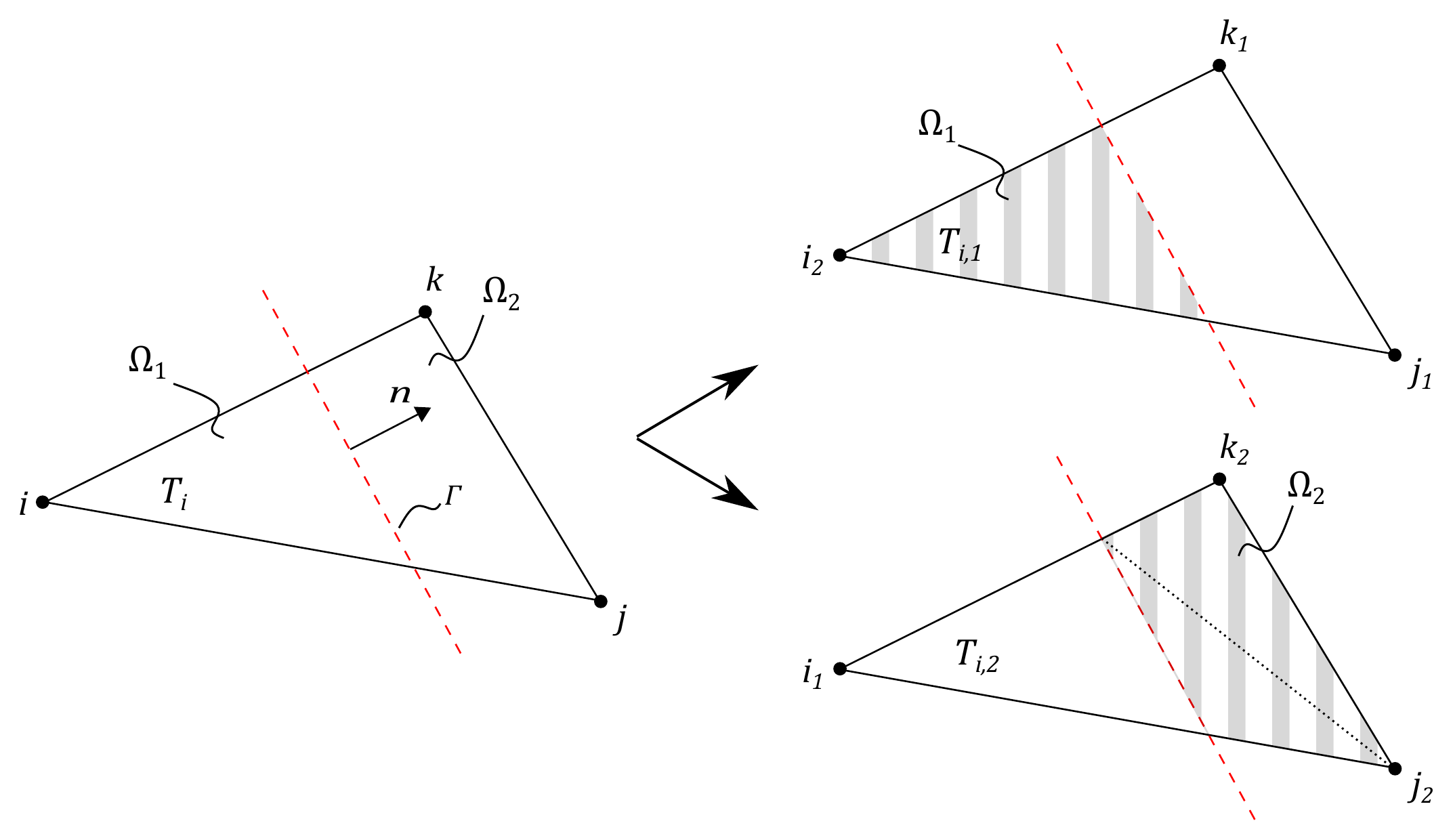}
	\end{center}
	\caption{The split of a triangle \textit{without} bifurcation point.}
	\label{TriSplit1}
\end{figure}

\begin{figure}
	\begin{center}
		\includegraphics[scale=0.4]{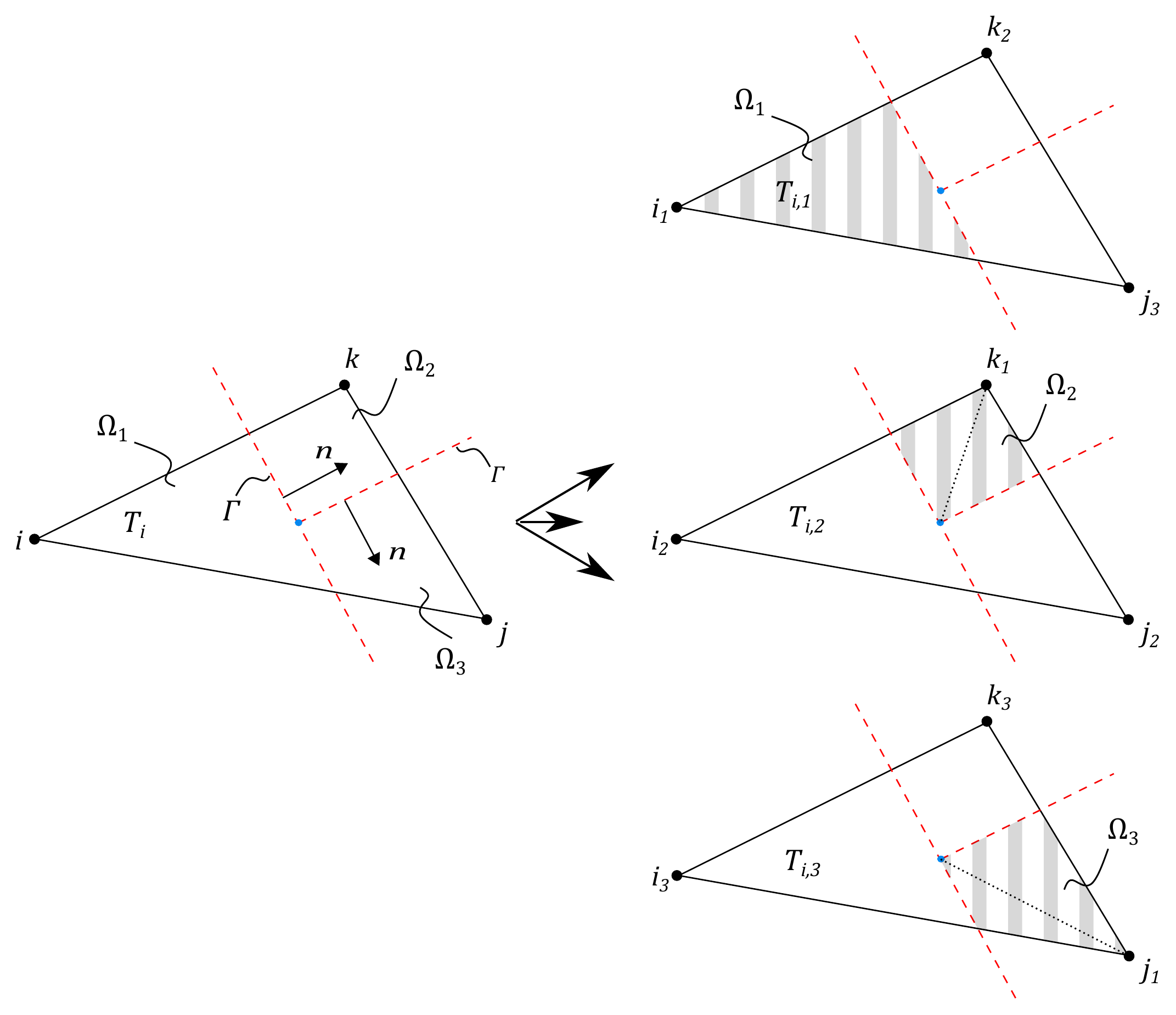}
	\end{center}
	\caption{The split of a triangle \textit{with} bifurcation point.}
	\label{TriSplit2}
\end{figure}

\subsection{Example 1. No Flow in Fracture}
We consider an example on $\Omega  = (0,1) \times (0,1)$, from \cite{HaHa02}. We solved 
the example with an added bifurcation point. For the added fracture, we denote the diffusion coefficient by $a_{\Gamma_1}$. The exact solution is given by
\begin{equation}
u(x, y) = 
\begin{dcases}
\frac{r^2}{a_1}, 
& \text{if } r \leqslant r_0 \\
\frac{r^2}{a_2} - \frac{r_0^2}{a_2} + \frac{r_0^2}{a_1}, 
& \text{if } r > r_0
\end{dcases} 
\end{equation}
where \(r = \sqrt{x^2 + y^2}\). We chose \(r_0 = 3/4  \), \(a_1 = 1  \), \(a_2 = 1000  \) and \(a_\Gamma = a_{\Gamma_1} = 0  \), with a right-hand side \(f = -4\) and \(f_\Gamma = 0\). The boundary conditions were symmetry boundaries at \(x = 0\) and \(y = 0\) and Dirichlet boundary conditions corresponding to the exact solution at \(x = 1\) and \(y = 1\). This example is outlined in Figure \ref{ex1_outline} and Figure \ref{stabTerm}. We give the elevation of the approximate solution in Figure \ref{ex1_sol}, on the last mesh in a sequence. The corresponding convergence of the \( L_2\)-norm and the energy-norm is given in Figure \ref{ex1_conv}. 

\begin{figure}
	\begin{center}
		\includegraphics[scale=0.5]{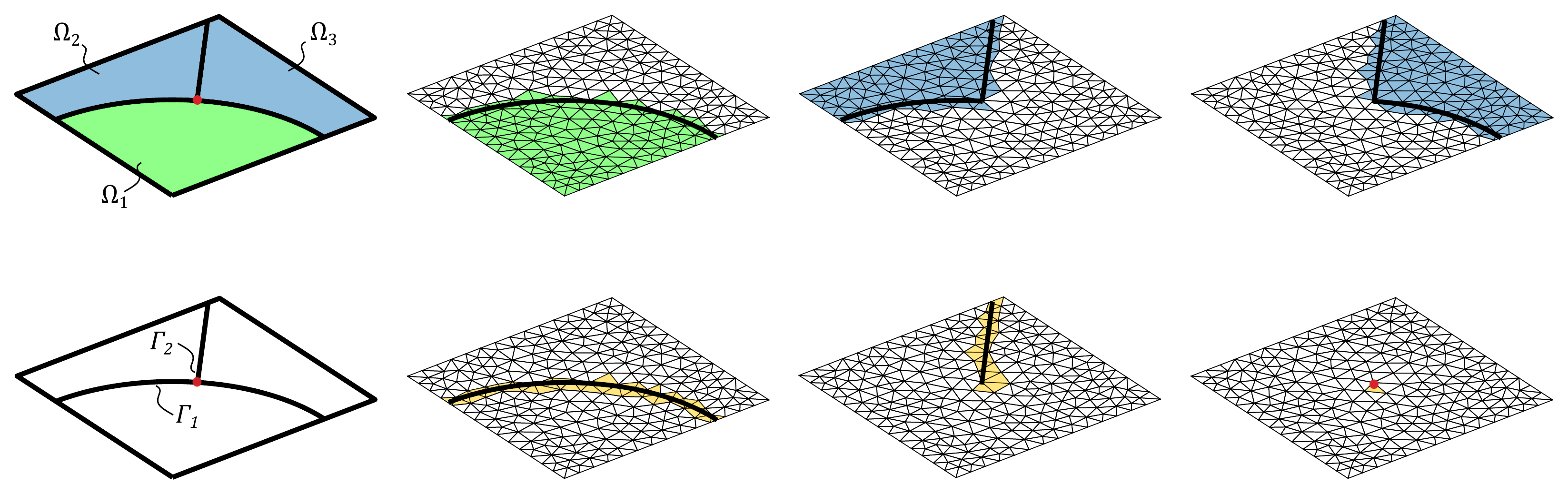}
	\end{center}
	\caption{Active meshes with two embedded fractures, \textit{Example 1}.}
	\label{ex1_outline}
\end{figure}

\begin{figure}
	\begin{center}
		\includegraphics[scale=0.7]{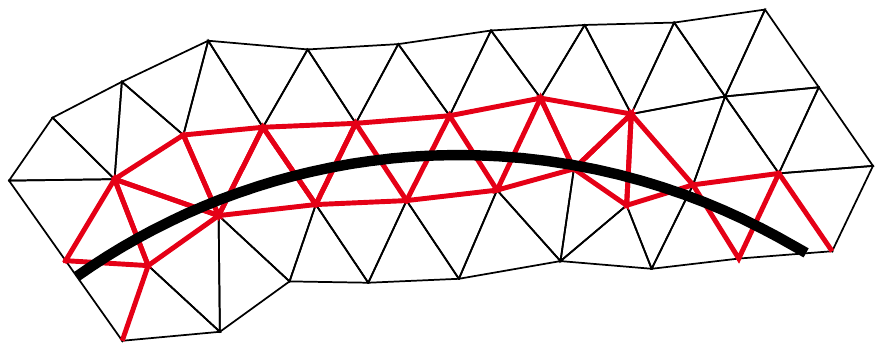}
	\end{center}
	\caption{The red edges indicates the selection for computing stabilization terms asscording to (\ref{eq:stab}). \label{fig:edges}}
	\label{stabTerm}
\end{figure}

\begin{figure}
	\begin{center}
		\includegraphics[scale=0.5]{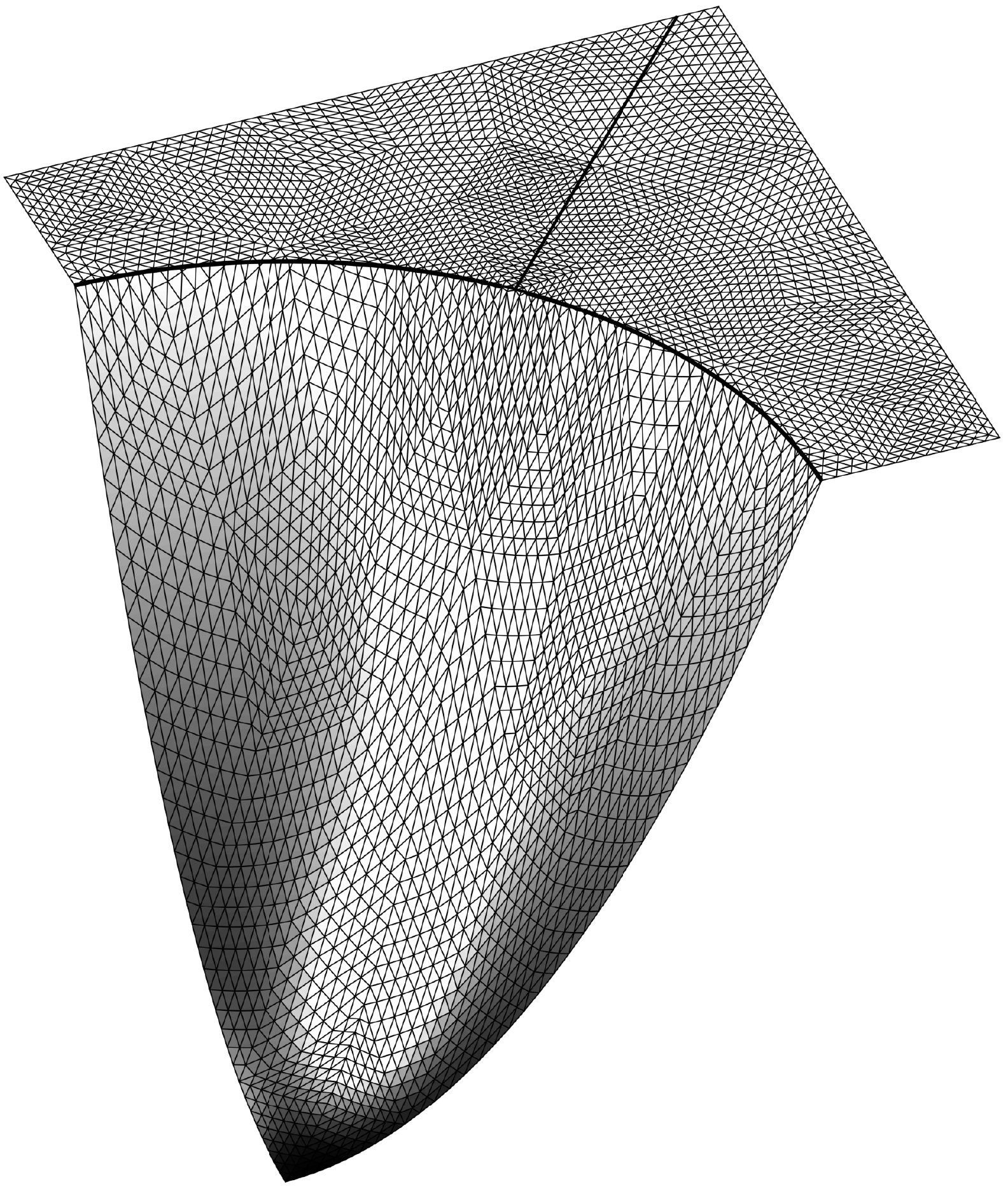}
	\end{center}
	\caption{Elevation of the approximate solution with two embedded fractures, \textit{Example 1}.}
	\label{ex1_sol}
\end{figure}

\begin{figure}
	\begin{center}
		\includegraphics[scale=0.8]{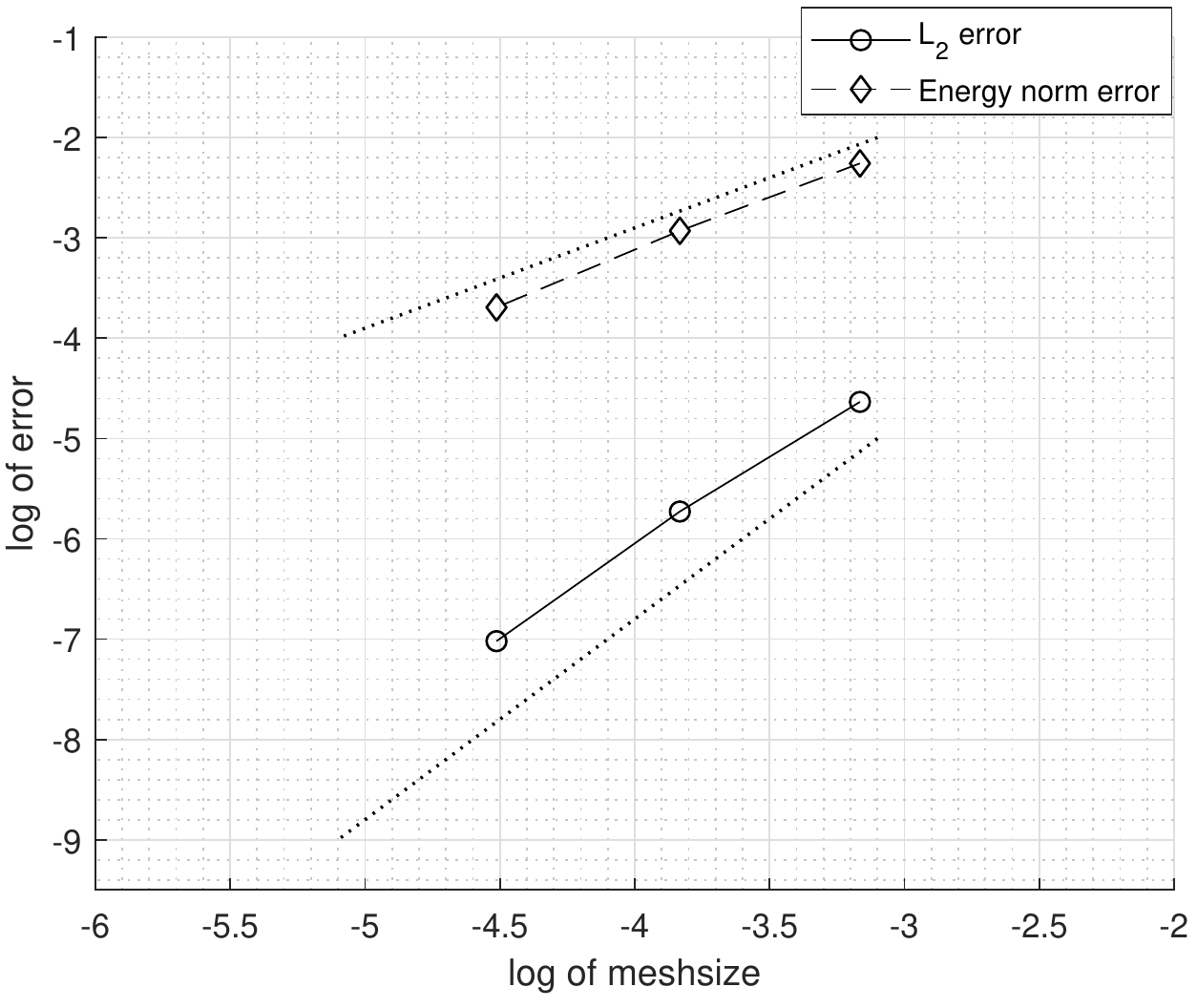}
	\end{center}
	\caption{\( L_2\)-norm and energy-norm convergence using natural logarithm with two embedded fractures, \textit{Example 1}. Dotted lines signify optimal convergence. Inclination 1:1 for energy-norm and 2:1 for \( L_2\)-norm. }
	\label{ex1_conv}
\end{figure}

\subsection{Example 2. Flow in the Fracture}
We considered a two-dimensional example on the domain \(\Omega \) = (1, \( e^{5/4} \)) \(\times \) (1, \( e^{5/4} \)), from \cite{BuHaLa17b}. 
We solved the example with an additional fracture added, see Figure \ref{ex2_outline}. 
The exact solution is given by
\[
\begin{split}
u_1 & = \frac{\log (r)}{5} (4 + e) \quad \text{for } \quad 1 < r < e, \\
u_2 & = \frac{4 - 4e}{5} (\log(r) - \frac{5}{4}) + 1 \quad \text{for } \quad e < r <  e^{5/4},
\end{split}
\]

\noindent where \(\sqrt{x^2 + y^2} := r = e\). We chose \(a_1 = a_2 = a_\Gamma = 1  \) and the right hand side to \(f = f_\Gamma = 0\). For the added crack we chose \(a_{\Gamma_1} = 0 \). The Dirichlet boundary conditions corresponding to the exact solution at \(x,y = 0\) and \(x,y = 1\). In Figure \ref{ex2_sol}., we give the elevation of the approximate solution. The corresponding \( L_2\)-norm convergence and the energy-norm is given in Figure \ref{ex2_conv}. 

\begin{figure}
	\begin{center}
		\includegraphics[scale=0.5]{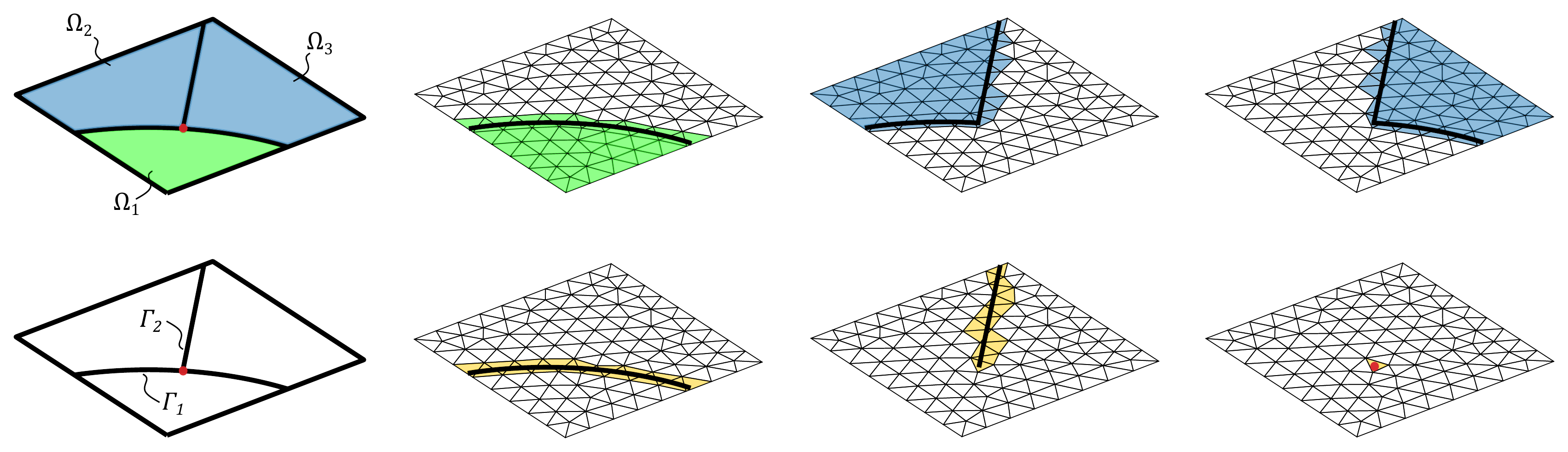}
	\end{center}
	\caption{Active meshes with two embedded fractures, \textit{Example 2}.}
	\label{ex2_outline}
\end{figure}

\begin{figure}
	\begin{center}
		\includegraphics[scale=0.4]{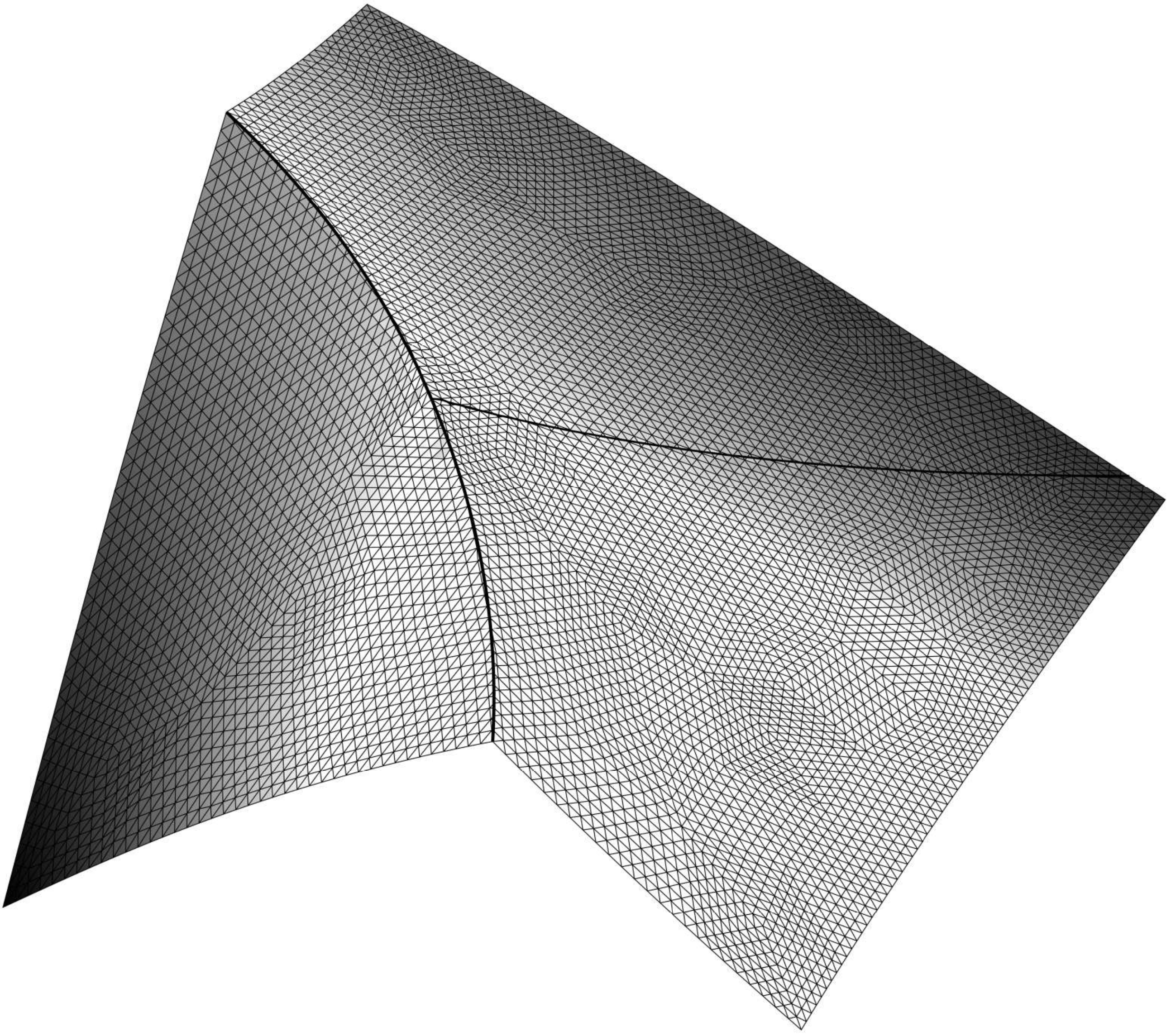}
	\end{center}
	\caption{Elevation of the approximate solution with two embedded fractures, \textit{Example 2}.}
	\label{ex2_sol}
\end{figure}

\begin{figure}
	\begin{center}
		\includegraphics[scale=0.8]{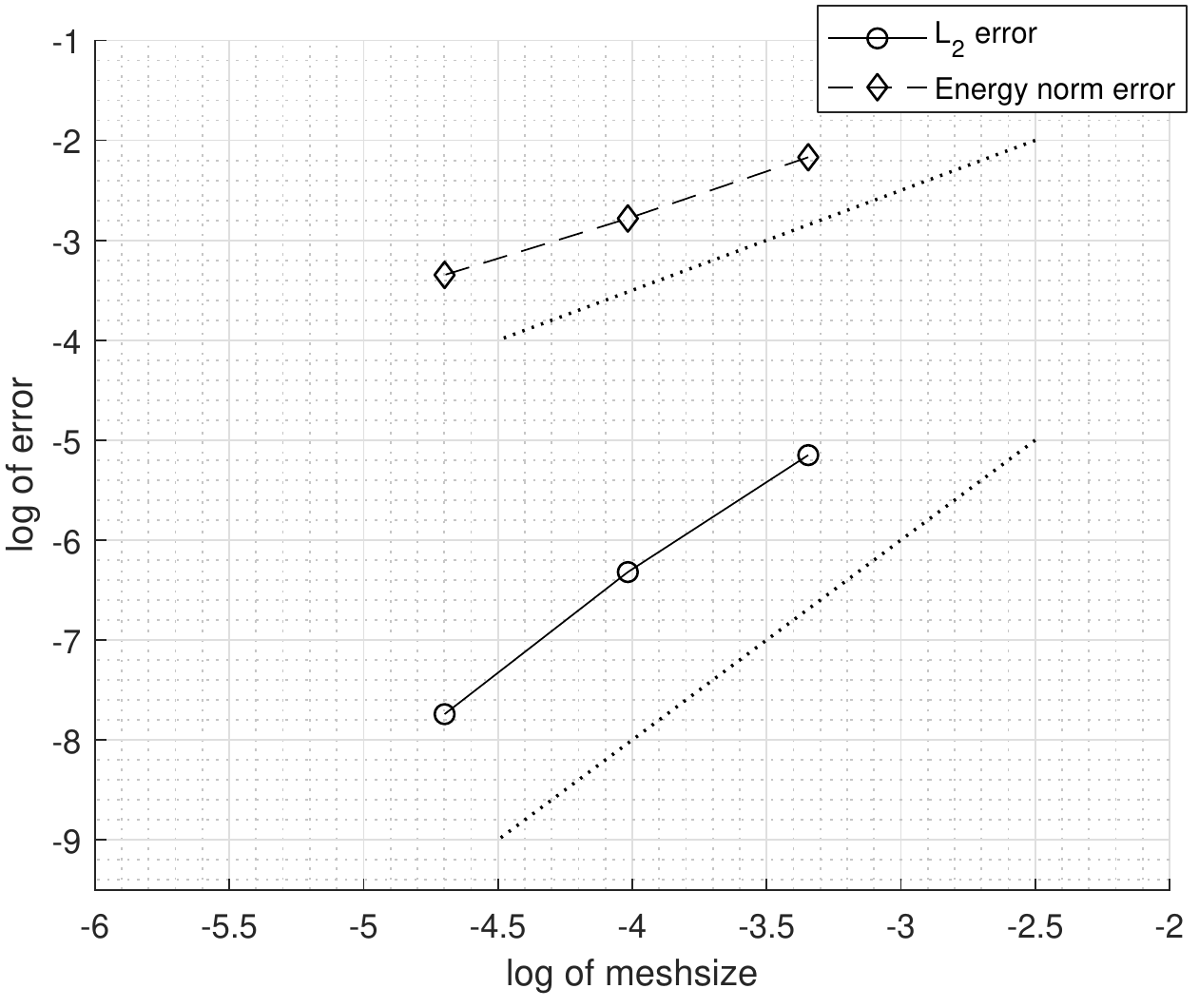}
	\end{center}
	\caption{\( L_2\)-norm and energy-norm convergence using natural logarithm with two embedded fractures,  \textit{Example 2}. Dotted lines signify optimal convergence. Inclination 1:1 for energy-norm and 2:1 for \( L_2\)-norm.}
	\label{ex2_conv}
\end{figure}

\subsection{Example 3. Flow in Bifurcating Fractures}
We consider an example with two bifurcating points. The fractures are modeled using higher order curves. In Figure \ref{ex3_outline} we show the fractures and construction of individual elements. On the domain \(\Omega \) = (0, 1) \(\times \) (0, 1), we chose \( a_1 = a_2 = 1\), \( f_\Omega = 1  \) and \(f_{\Gamma} = 0  \). We impose the Dirichlet boundary conditions \( u = 0 \) at \(x,y = 0 \) and \( u = 1 \) at \(x,y = 1 \). For the diffusion coefficient, we denote \( a_{\Gamma_i}  \) for each fracture and assign an individual value for each \( \Gamma_i \), see Figure \ref{ex3_outline_gamma}. In Figure \ref{ex3_solAB} through Figure \ref{ex3_solEF}, we present the solutions using global refinement with \( a_{\Gamma_i} \in \{0, 100\} \).

\begin{figure}
	\begin{center}
		\includegraphics[scale=0.5]{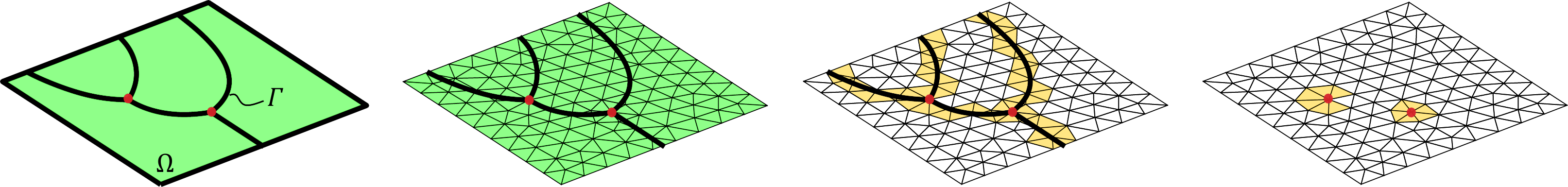}
	\end{center}
	\caption{Active meshes with two bifurcating points, \textit{Example 3}.}
	\label{ex3_outline}
\end{figure}

\begin{figure}
	\begin{center}
		\includegraphics[scale=1]{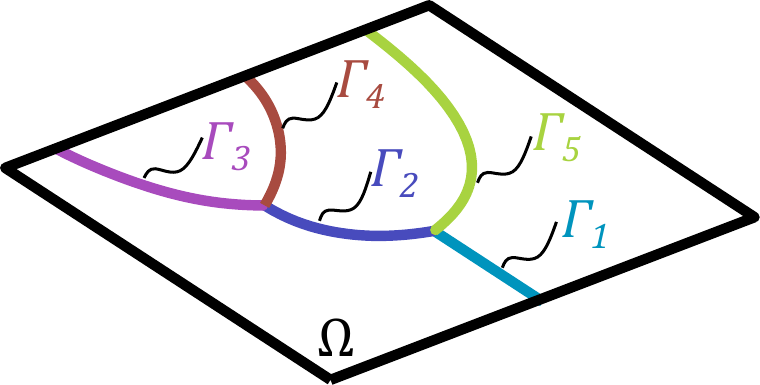}
	\end{center}
	\caption{Embedded fractures with assigned \( \Gamma  \), \textit{Example 3}.}
	\label{ex3_outline_gamma}
\end{figure}

\begin{figure}
	\begin{center}
		\includegraphics[scale=0.2]{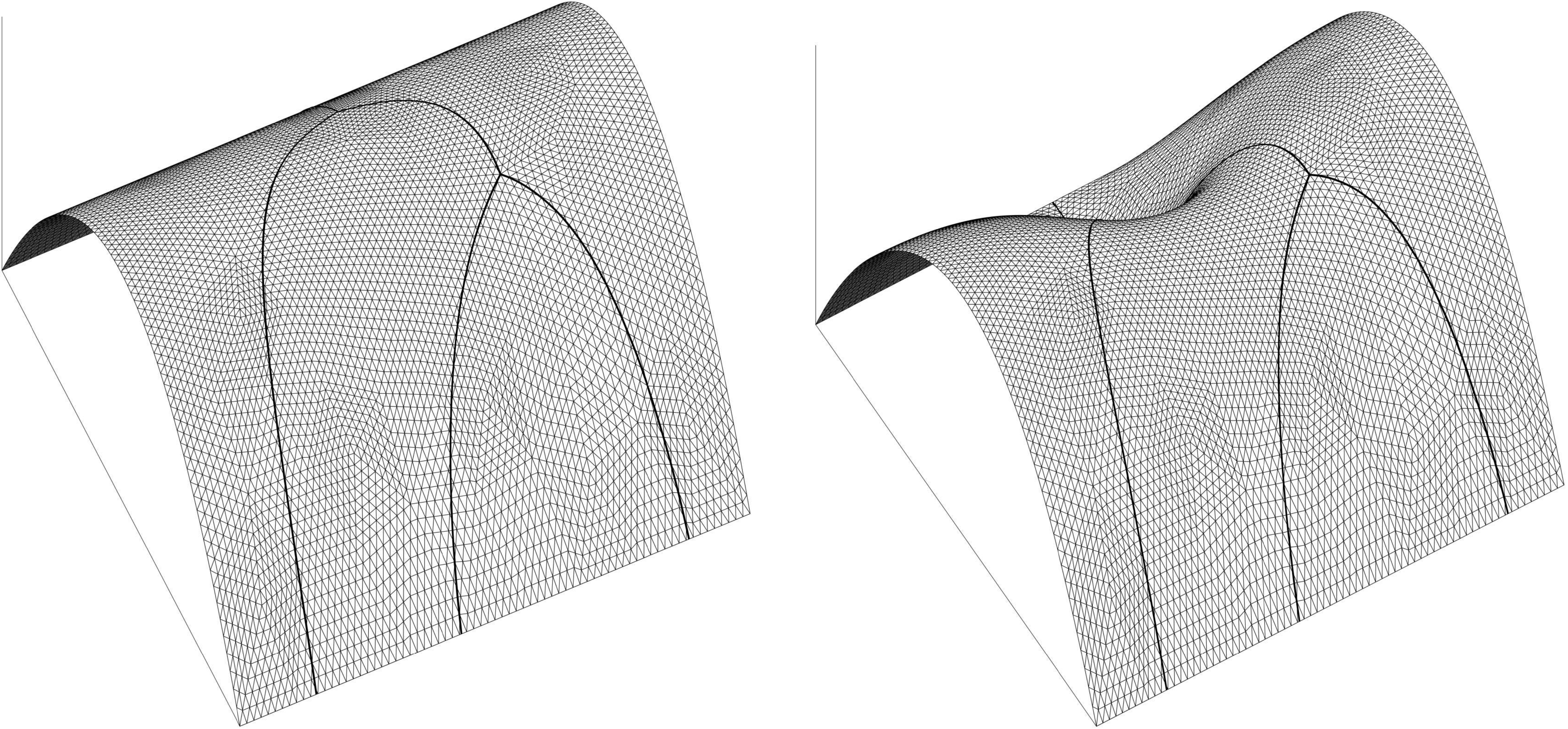}
	\end{center}
	\caption{Elevation of the approximate solution using two bifurcating points, \textit{Example 3}. Assigned value to the left figure: \( a_{\Gamma_1} = a_{\Gamma_2} = a_{\Gamma_3} = a_{\Gamma_4} = a_{\Gamma_5} = 0 \),  and assigned values to the right figure: \( a_{\Gamma_1} = 100\) and \( a_{\Gamma_2} = a_{\Gamma_3} = a_{\Gamma_4} = a_{\Gamma_5} = 0 \).}
	\label{ex3_solAB}
\end{figure}

\begin{figure}
	\begin{center}
		\includegraphics[scale=0.2]{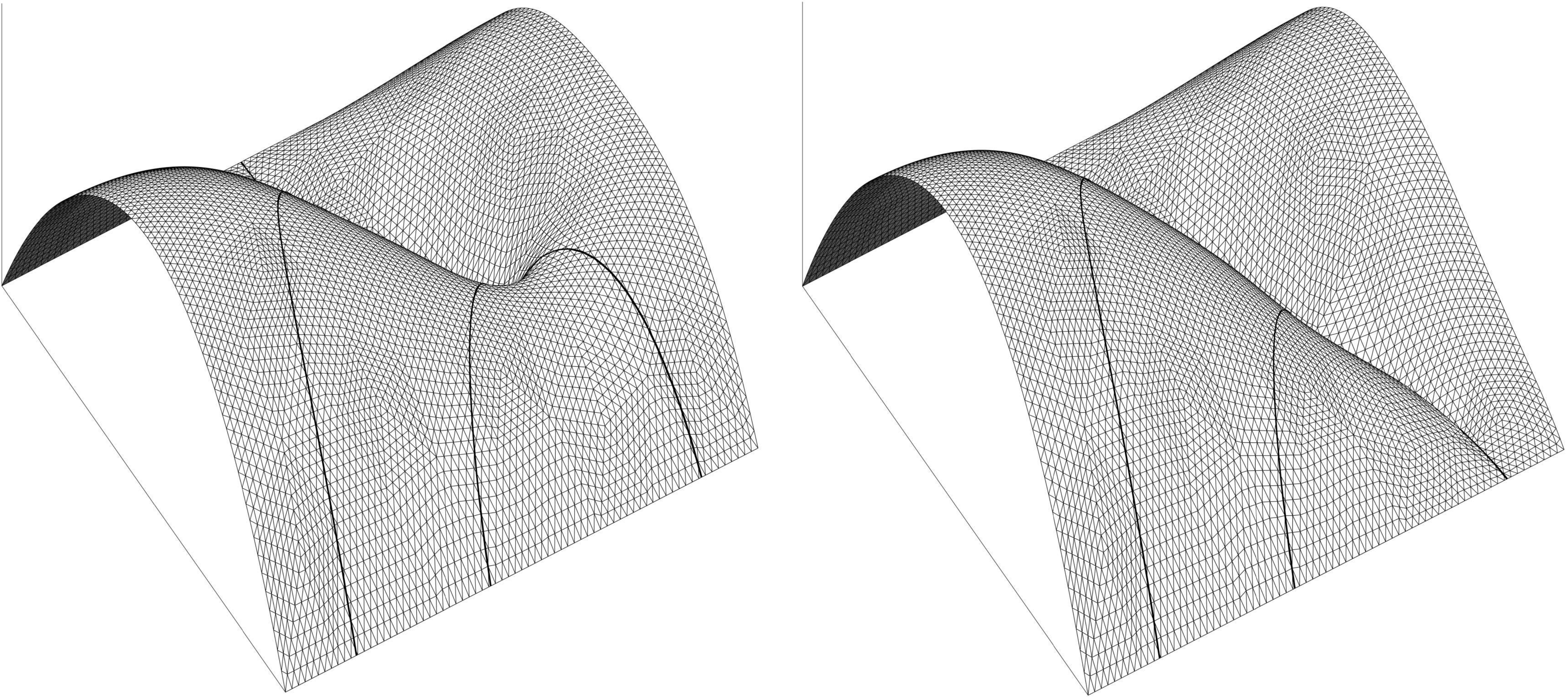}
	\end{center}
	\caption{Elevation of the approximate solution using two bifurcating points, \textit{Example 3}. Assigned value to the left figure: \( a_{\Gamma_1} = a_{\Gamma_2} = 100 \) and \( a_{\Gamma_3} = a_{\Gamma_4} = a_{\Gamma_5} = 0 \),  and assigned values to the right figure: \( a_{\Gamma_1} = a_{\Gamma_2} = a_{\Gamma_3} = 100 \) and \( a_{\Gamma_4} = a_{\Gamma_5} = 0 \).}
	\label{ex3_solCD}
\end{figure}

\begin{figure}
	\begin{center}
		\includegraphics[scale=0.2]{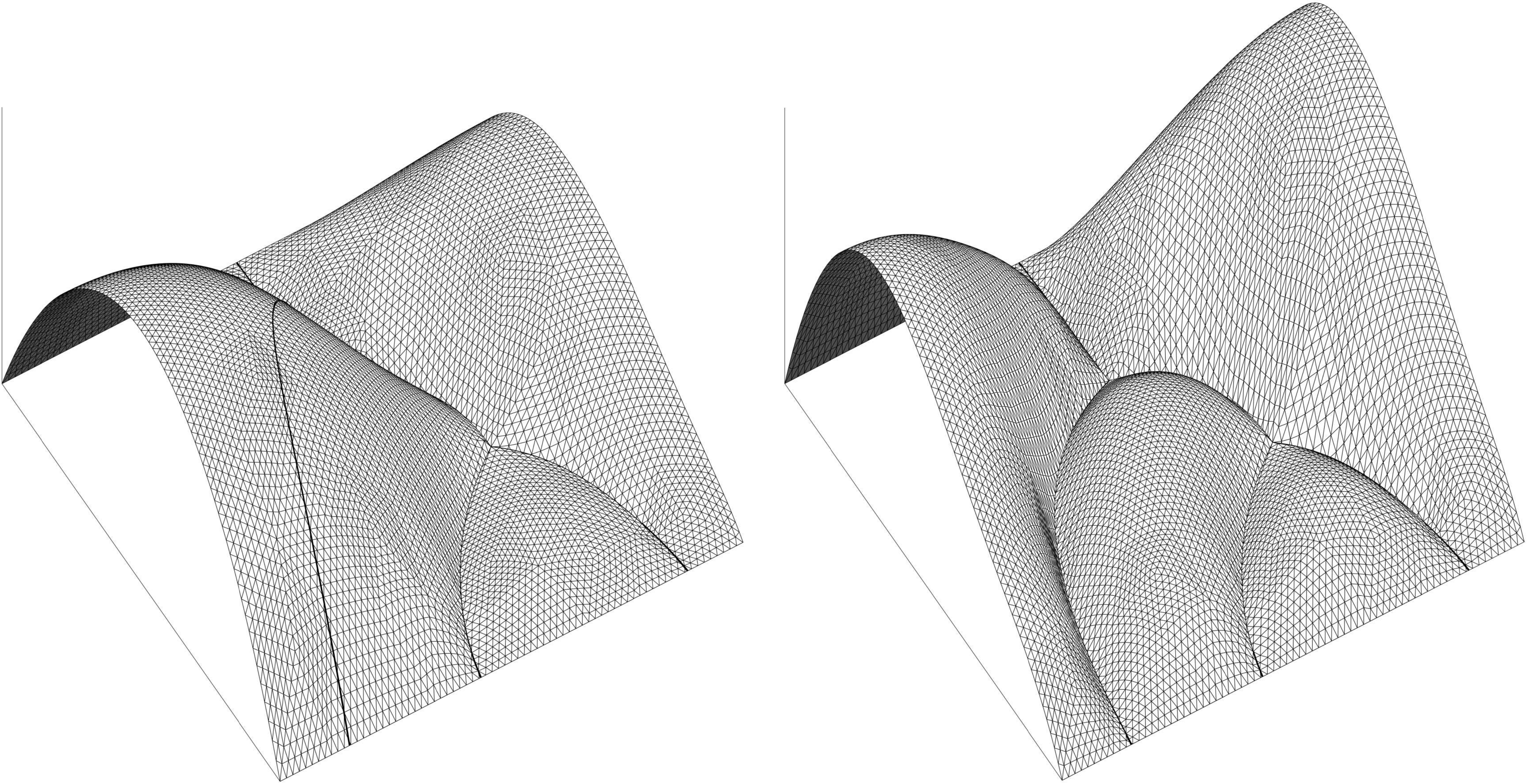}
	\end{center}
	\caption{Elevation of the approximate solution using two bifurcating points, \textit{Example 3}. Assigned value to the left figure: \( a_{\Gamma_1} = a_{\Gamma_2} = a_{\Gamma_3} = a_{\Gamma_4} = 100 \) and \(a_{\Gamma_5} = 0 \),  and assigned values to the right figure: \( a_{\Gamma_1} = a_{\Gamma_2} = a_{\Gamma_3} = a_{\Gamma_4} = a_{\Gamma_5} = 100 \).}
	\label{ex3_solEF}
\end{figure}

\bibliographystyle{abbrv}
\footnotesize{
\bibliography{Embedded}
}

\bigskip
\paragraph{Acknowledgements.}
This research was supported in part by the Swedish Foundation
for Strategic Research Grant No.\ AM13-0029, the Swedish Research
Council Grants Nos.\  2013-4708, 2017-03911, and the Swedish
Research Programme Essence. EB was supported in part by the EPSRC grant EP/P01576X/1.

\bigskip
\bigskip
\noindent
\footnotesize {\bf Authors' addresses:}

\smallskip
\noindent
Erik Burman,  \quad \hfill \addressuclshort\\
{\tt e.burman@ucl.ac.uk}

\smallskip
\noindent
Peter Hansbo,  \quad \hfill \addressjushort\\
{\tt peter.hansbo@ju.se}

\smallskip
\noindent
Mats G. Larson,  \quad \hfill \addressumushort\\
{\tt mats.larson@umu.se}

\smallskip
\noindent
David Samvin, \quad \hfill \addressjushort\\
{\tt david.samvin@ju.se}

\end{document}